\documentclass{article}

\usepackage{arxiv}

\usepackage[utf8]{inputenc}
\usepackage{amsthm}
\usepackage{amsmath, amssymb}
\usepackage{bbold}
\usepackage{graphicx}
\usepackage{subcaption}
\usepackage{tkz-euclide}
\usepackage[numbers]{natbib}
\usepackage{mathtools}
\usepackage{xfrac}  
\usepackage{float}
\usepackage{tikz-cd} 
\usepackage{varwidth}
\usepackage{tasks}
\usetikzlibrary{decorations.pathreplacing,shapes.misc}
\usetikzlibrary{3d,shapes.geometric,positioning,intersections}

\usepackage{physics}

\usepackage{tikz-3dplot}
\usetikzlibrary{arrows.meta, calc}
\usepackage{xcolor}
\usepackage{bm}
\usepackage[e]{esvect}
\usepackage{tkz-euclide}

\setcitestyle{square}

\usetikzlibrary{arrows}
\usepackage    {amsmath}
\usepackage    {lipsum}
\usepackage    {tikz}
\usetikzlibrary{calc,matrix}
\usepackage{nicematrix}

\DeclareMathOperator{\rk}{\text{rk}}

\DeclareMathOperator{\im}{\text{im}}

\newtheorem{thm}{Theorem}
\newtheorem{lem}[thm]{Lemma}

\newtheorem{defn}[thm]{Definition}

\newcommand{\N}{\mathbb{N}}

\newcommand{\R}{\mathbb{R}}

\newcommand{\hypermatrix}[8]
{
  \begin{tikzpicture}[baseline=(center),every node/.style={anchor=north east,minimum width=7mm,minimum height=5mm}]
  \matrix (mA) [draw,matrix of math nodes,ampersand replacement=\&]
  {
    #1 \& #2\\
    #3 \& #4\\
  };
  \matrix (mB) [draw,matrix of math nodes,ampersand replacement=\&] at ($(mA.south west)+(2.6,2.6)$)
  {
    #5 \& #6\\
    #7 \& #8\\
  };
  \draw[dashed](mA.north west)--(mB.north west);
  \draw[dashed](mA.south east)--(mB.south east);
  \node (center) at ($(mA.south east)!0.5!(mB.north west)$) {};
  \end{tikzpicture}
}

\newcommand{\hypermatrixlong}[8]
{
  \begin{tikzpicture}[baseline=(center),every node/.style={anchor=north east,minimum width=7mm,minimum height=5mm}]
  \matrix (mA) [draw,matrix of math nodes,ampersand replacement=\&]
  {
    #1 \& #2\\
    #3 \& #4\\
  };
  \matrix (mB) [draw,matrix of math nodes,ampersand replacement=\&] at ($(mA.south west)+(3.3,3.3)$)
  {
    #5 \& #6\\
    #7 \& #8\\
  };
  
  \draw[dashed](mA.north west)--(mB.north west);
  \draw[dashed](mA.south east)--(mB.south east);
  \node (center) at ($(mA.south east)!0.5!(mB.north west)$) {};
  \end{tikzpicture}
}

\usepackage[utf8]{inputenc} 
\usepackage[T1]{fontenc}    
\usepackage{url}            
\usepackage{booktabs}       
\usepackage{amsfonts}       
\usepackage{nicefrac}       
\usepackage{microtype}      
\usepackage{lipsum}
\usepackage{graphicx}
\graphicspath{ {./images/} }

\title{Rank Two Approximations of $2 \times 2 \times 2$ Tensors over $\mathbb{R}$}

\author{David Warren Katz}

\begin{document}

\date{}
\maketitle

\begin{abstract}
We provide a coordinate-free proof that real $2 \times 2 \times 2$ rank three tensors do not have optimal rank two approximations with respect to the Frobenius norm. This result was first proved in \cite[Thm. 8.1]{Lim} by considering the ${\text{GL}(V^1) \times \text{GL}(V^2) \times \text{GL}(V^3)}$ orbit classes of ${V^1 \otimes V^2 \otimes V^3}$ and the $2 \times 2 \times 2$ hyperdeterminant. Our coordinate-free proof expands on the result in \cite[]{Lim} by developing a proof method that can be generalized more readily to higher dimensional $n_1 \times n_2 \times n_3$ tensor spaces.
\end{abstract}

\section{Introduction}

Let $V^1$, $V^2$, and $V^3$ be two-dimensional real vector spaces, respectively, and let ${V^1 \otimes V^2 \otimes V^3}$ denote the tensor product of these spaces. The tensors in the form $v^1 \otimes v^2 \otimes v^3$ for some vectors $v^1$, $v^2$, and $v^3$ are called simple tensors. Every tensor can be written as the sum of finitely many simple tensors. The rank of a tensor $\tau$ is the minimum number $n$ such that $\tau$ is the sum of $n$ simple tensors. That is,
\begin{gather*}
    \text{rk} (\tau) \ = \ \min \{n \in \mathbb{N} \mid \ \tau \ = {\sum_{i = 1}^{n} \ v_{i}^1 \otimes v_{i}^2 \otimes v_{i}^3} \ \text{ for } \text{some } v_i^s \in V^s \}.
\end{gather*}
Once bases $\{e^s_t \}_{t = 1}^{2}$ of $V^s$ are chosen for $s=1,2,3$, a tensor in ${V^1 \otimes V^2 \otimes V^3}$ can be coordinitized as a hypermatrix in $\R^{2 \times 2 \times2}$ by the isomorphism from $V^1 \otimes V^2 \otimes V^3$ to $\R^{2 \times 2 \times 2}$ defined on simple tensors as
\begin{align*}
    (a^1_1 e^1_1 + a^1_2e^1_2) \otimes (a^2_1 e^2_1 + a^2_2e^2_2) \otimes (a^3_1 e^3_1 + a^3_2e^3_2) \ \  \longmapsto \ \     \hypermatrixlong{a_1^1a_1^2a_1^3}{a_1^1a_1^2a_2^3}{a_1^1a_2^2a_1^3}{a_1^1a_2^2a_2^3}{a_2^1a_1^2a_1^3}{a_2^1a_1^2a_2^3}{a_2^1a_2^2a_1^3}{a_2^1a_2^2a_2^3}
\end{align*}
for some real constants $a^s_t$. A lot of information about the rank of the $2 \times 2 \times 2$ real hypermatrix
\begin{gather*}
(a_{ijk}) \ = \ \hypermatrix{a_{111}}{a_{112}}{a_{121}}{a_{122}}{a_{211}}{a_{212}}{a_{221}}{a_{222}} 
\end{gather*}
can be inferred from sign of the polynomial
\begin{align*}
\Delta(a_{ijk}) &= (a_{111}^2a_{222}^2 + a_{112}^2a_{221}^2 + a_{121}^2 a_{212}^2 + a_{122}^2 a_{211}^2 ) \\
& \hspace{10mm} -2( a_{111}a_{112}a_{221}a_{222}+a_{111}a_{121}a_{212}a_{222}+ a_{111}a_{122}a_{211}a_{222}) \\
& \hspace{10mm} -2(a_{112}a_{121}a_{212}a_{221} + a_{112}a_{122}a_{221}a_{211}+ a_{121}a_{122}a_{212}a_{211}) \\
& \hspace{10mm} + 4(a_{111}a_{122}a_{212}a_{221} + a_{112}a_{121}a_{211}a_{222}).
\end{align*}
The polynomial $\Delta$ is called the $2 \times 2 \times2$ hyperdeterminant. $\Delta(a_{ijk}) > 0$ implies hypermatrix $(a_{ijk})$ is rank two, and $\Delta(a_{ijk}) < 0$ implies $(a_{ijk})$ is rank three. The existence of such a polynomial is not particular to the $2 \times 2 \times 2$ case. The sets of constant rank real hypermatrices are semialgebraic for any dimensions $n_1 \times n_2 \times \dots \times n_d$. Hence, the rank of a real hypermatrix can always be computed by evaluating a finite number of polynomial equalities and inequalities, where the variables of the polynomials are the entries of the hypermatrix. Because of this, real tensors are usually studied in coordinates by representing tensors as hypermatrices and then exploiting properties of these polynomials. The polynomials that define the sets of constant rank $2 \times 2 \times 2$ tensors are different, however, from the polynomials that define the sets of constant rank $n_1 \times n_2 \times n_3$ tensors for any other dimensions $n_1 \times n_2 \times n_3$. As a result, coordinate proofs about $2 \times 2 \times 2$ tensors that rely on these polynomials cannot be readily generalized to the arbitrary $n_1 \times n_2 \times n_3$ case. In this paper, we take an alternative coordinate-free approach in our study of $2 \times 2 \times 2$ tensors that lends itself more readily to generalization. We prove that real $2 \times 2 \times 2$ rank three tensors do not have optimal rank two approximations with respect to the Frobenius norm without relying on hypermatrices or the hyperdeterminant.

Matrix rank is lower semi-continuous for matrices of real numbers. That is, if a sequence of rank $r$ real matrices converges to a matrix $A$ in the norm topology, then the rank of $A$ is less than or equal to $r$. This guarantees the existence of optimal low-rank approximations of real matrices. A sequence of real hypermatrices, however, can converge to a hypermatrix of greater rank. As a result, optimal low rank approximations may not exist for real hypermatrices. In fact, the set of real hypermatrices with no optimal low rank approximation often has positive Lebesgue measure \cite{Alwin}, so characterizing such real hypermatrices is an important step in implementing algorithms with real hypermatrices.

Let $\tau \in V^1 \otimes V^2 \otimes V^3$ be a tensor of rank $r$, and let $s \le r$. An optimal rank $s$ approximation of $\tau$ with respect to norm $\| \cdot \|$ is a tensor $\upsilon$ such that 
     \begin{equation*}
         \| \tau - \upsilon \| \ =  \ \inf_{\substack{\upsilon' \in V^1 \otimes V^2 \otimes V^3 \\ \rk(\upsilon') \le s}}\ \| \tau - \upsilon' \|.
     \end{equation*}     
 
\noindent It is a classical observation that tensors in the form
\begin{gather} 
  \beta \ = \ x^1_2 \otimes x^2_1 \otimes  x_1^3 \ + \ x_1^1 \otimes  x_2^2 \otimes  x_1^3 \ + \ x_1^1 \otimes  x_1^2 \otimes  x_2^3 \label{tangent_form_1}
\end{gather}  
are limit points of the sequence of rank at most two tensors 
\begin{equation}
    \tau_n \ = \ n \ (x_1^1 + \frac{1}{n} x_2^1) \otimes  (x_1^2 + \frac{1}{n} x_2^2) \otimes    (x_1^3 + \frac{1}{n} x_2^3) \ - \ n x_1^1 \otimes  x_1^2 \otimes x_1^3, \ n \in \N \label{sequence_rank_two}
\end{equation}
in the norm topology. It follows from the triangle inequality and the multilinearity of the tensor product that
\begin{gather*}
    \norm{\tau_n - \beta} \ \leq \ \frac{1}{n} \ \norm{x_2^1 \otimes x_2^2 \otimes x_1^3 \ + \ x_2^1 \otimes x^2_1 \otimes x_2^3 \ + \ x_1^1 \otimes x_2^2 \otimes x_2^3} \ + \ \frac{1}{n^2} \ \norm{x_2^1 \otimes x_2^2 \otimes x_2^3}.
\end{gather*}
Hence, $\lim\limits_{n \to \infty}  \Vert \tau_n - \beta \Vert = 0$, so $\lim\limits_{n \to \infty} \tau_n$ indeed equals $\beta$. In Section \ref{section_2}, we give a coordinate-free proof that tensors in the form of (\ref{tangent_form}) are in fact rank three when the sets $\{x_1^1, x_2^1\}$, $\{x_1^2, x_2^2\}$, and $\{x_1^3, x_2^3\}$ are independent. Thus, when $\{x_1^1, x_2^1\}$, $\{x_1^2, x_2^2\}$, and $\{x_1^3, x_2^3\}$ are independent, $\beta$ is an example of a rank three tensor with no an optimal rank two approximation. The main result of this paper is a coordinate-free proof that in fact $\it{every}$ rank three $2 \times 2 \times 2$ tensor has no optimal rank two approximation with respect to the Frobenius norm, not just the tensors in the form of (\ref{tangent_form_1}). 

This result was first proved in coordinates in \cite[Thm. 8.1]{Lim}. The argument in \cite[]{Lim} can be summarized as follows. Suppose for contradiction that the $2 \times 2 \times 2$ real hypermatrix $B$ is an optimal rank two approximation of rank three $2 \times 2 \times 2$ real hypermatrix $A$ with respect to the Frobenius norm. By considering properties of the polynomial $\Delta$, it follows that $\Delta(B) = 0$. The polynomial $\Delta$ is invariant on the ${\text{GL}(V^1) \times \text{GL}(V^2) \times \text{GL}(V^3)}$ orbit classes of ${V^1 \otimes V^2 \otimes V^3}$, and only three of the eight orbit classes are zero on $\Delta$. Hypermatrices in these three orbit classes are equivalent up to an orthogonal change of coordinates to hypermatrices in the form
\begin{gather}\label{special_form}
\hypermatrix{\lambda}{0}{0}{\mu}{0}{0}{0}{0} 
\end{gather}
for some real $\lambda$ and $\mu$. Since $B$ is rank two, we may thus assume $B$ is in form (\ref{special_form}) with both $\lambda$ and $\mu$ nonzero. Finally, it is shown that if $H$ is a $2 \times 2 \times 2$ hypermatrix such that 
\begin{gather}
    \Delta(B + \epsilon H) = 0 \label{condition_star}
\end{gather}
for all real $\epsilon$, then $A - B$ is orthogonal to $H$. By considering various hypermatrices $H$ that satisfy (\ref{condition_star}), the authors of \cite{Lim} then conclude that
\begin{gather*}
A - B \ = \ \hypermatrix{0}{0}{0}{0}{a\mu}{0}{0}{-a\lambda},
\end{gather*}
for some constant $a$. This implies that
\begin{gather*}
A \ = \ \hypermatrix{\lambda}{0}{0}{\mu}{a\mu}{0}{0}{-a\lambda}.
\end{gather*}
However, this hypermatrix is rank two, contradicting that $A$ is rank three. 

Our coordinate-free proof expands on the result in \cite[]{Lim} by developing a proof method that can be generalized more readily to higher dimensional tensor spaces. Our proof is also a proof by contradiction. We suppose for contradiction that rank two tensor $\upsilon$ in ${V^1 \otimes V^2 \otimes V^3}$ is an optimal rank two approximation of rank three tensor $\tau$. By considering the relationship between the mode-$1$ contraction maps of $\tau$ and $\upsilon$, we also derive the contradiction that $\tau$ is rank two. Our proof has the interesting geometric corollary that the nearest point of a rank three tensor to the second secant variety of the Segre variety is a rank three tensor in the tangent space of the Segre variety.

\section{Contraction Maps of Tensors}\label{section_2}

Let $V^1$, $V^2$, and $V^3$ be finite $n_1$, $n_2$, and $n_3$-dimensional real vector spaces, and let $V^{1*}$, $V^{2*}$, and $V^{3*}$ denote their dual spaces. The tensor ${\tau = \sum_{i =1}^r v^1_i \otimes v^2_i \otimes v^3_i \in V^1 \otimes V^2 \otimes V^3}$ induces the three linear maps 
\begin{align*}
    V^{1*} &\xrightarrow{\Pi_1(\tau)} V^2 \otimes V^3  &    V^{2*} &\xrightarrow{\Pi_2(\tau)} V^1 \otimes V^3  & & & V^{3*} &\xrightarrow{\Pi_3(\tau)} V^1 \otimes V^2 \\
    v^{1*} &\mapsto \sum_{i=1}^r \ v^{1*}(v^1_i) \ v^2_i \otimes v^3_i,  &   v^{2*} &\mapsto \sum_{i=1}^r \ v^{2*}(v^2_i) \ v^1_i \otimes v^3_i, &&\text{and} &   v^{3*} &\mapsto \sum_{i=1}^r \ v^{3*}(v^3_i) \ v^1_i \otimes v^2_i.
\end{align*}
\noindent These maps are called the mode-$1$, mode-$2$, and mode-$3$ contraction maps of $\tau$, respectively. 

\begin{thm}\label{contraction_maps_well_defined}
The mode-$i$ contraction map $\Pi_i$ is well-defined. 
\end{thm}
\begin{proof}
Suppose that ${\sum_{t=1}^r v^1_t \otimes v^2_t \otimes v^3_t}$ ${= \sum_{t=1}^s w^1_t \otimes w^2_t \otimes w^3_t}$ for vectors $v^i_t$, $w^i_t$ $\in V^i$. We need to show that 
\begin{equation*}
    \Pi_i \left( \sum_{t=1}^r v^1_t \otimes v^2_t \otimes v^3_t \right)(v^{i*}) = \Pi_i \left( \sum_{t=1}^s w^1_t \otimes w^2_t \otimes w^3_t \right) (v^{i*}) 
\end{equation*} 
for every $v^{i*} \in V^{i*}$ for $i = 1, 2$ and $3$. Choose bases $\{ e^i_t\}_{t=1}^{n_i}$ of $V^i$ with corresponding dual bases $\{ e^{i*}_t\}_{t=1}^{n_i}$ for each $i$. It is sufficient to show that \begin{equation*}
    \Pi_1 \left(\sum_{t=1}^r v^1_t \otimes v^2_t \otimes v^3_t \right) (e^{1*}_u) = \Pi_1 \left( \sum_{t=1}^s w^1_t \otimes w^2_t \otimes w^3_t \right) (e^{1*}_u) 
\end{equation*} 
for each $e^{1*}_u$ in our dual basis. Let $a^i_{t,u}$ and $b^i_{t,u}$ be scalars such that ${v^i_t = \sum_{u=1}^{n_i} a^i_{t,u} e^i_u}$ and ${w^i_t = \sum_{u=1}^{n_i} b^i_{t,u} e^i_u}$ for all $i$ and $t$. It follows that $\sum_{t=1}^r v^1_t \otimes v^2_t \otimes v^3_t = \sum_{t,i,j,k =1}^{r, n_1, n_2, n_3} a^1_{ti}a^2_{tj}a^3_{tk} \ e^1_i \otimes e^2_j \otimes e^3_k$, which must equal $\sum_{t=1}^s w^1_t \otimes w^2_t \otimes w^3_t = \sum_{t,i,j,k =1}^{s, n_1, n_2, n_3} b^1_{ti}b^2_{tj}b^3_{tk} \ e^1_i \otimes e^2_j \otimes e^3_k$, so $\sum_{t=1}^{r} a^1_{ti}a^2_{tj}a^3_{tk}$ $= \sum_{t=1}^{s} b^1_{ti}b^2_{tj}b^3_{tk}$ for all $i,j,k$. Hence, 
\begin{align*}
    \Pi_1 \left(\sum_{t=1}^r v^1_t \otimes v^2_t \otimes v^3_t \right)(e^{1*}_u) &=  \sum_{t,j,k =1}^{r, n_2, n_3} a^1_{tu}a^2_{tj}a^3_{tk} \ e^2_j \otimes e^3_k \\
     &=  \sum_{j,k =1}^{n_2, n_3} \left( \sum_{t=1}^r a^1_{tu}a^2_{tj}a^3_{tk} \right) \  e^2_j \otimes e^3_k \\
     &=  \sum_{j,k =1}^{n_2, n_3} \left( \sum_{t=1}^s b^1_{tu}b^2_{tj}b^3_{tk} \right) \  e^2_j \otimes e^3_k \\
     &=  \sum_{t,j,k =1}^{s, n_2, n_3} b^1_{tu}b^2_{tj}b^3_{tk} \ e^2_j \otimes e^3_k \\
     &= \Pi_1 \left( \sum_{t=1}^s w^1_t \otimes w^2_t \otimes w^3_t \right)(e^{1*}_u).
\end{align*}
\end{proof}

The relationship between the contraction maps of a tensor generalizes, in a coordinate-free way, the fundamental relationship between the rows and columns of a matrix to hypermatrices. The rank of a two-fold tensor $\tau \in V^1 \otimes V^2$ is equal to the dimension of the  image of $\Pi_1(\tau)$, which is also equal to the dimension of the image of $\Pi_2(\tau)$. We now use contraction maps to show when tensors in the form (\ref{tangent_form_1}) are rank three. First, we need the following lemma.

\begin{lem}\label{independentunique}
Let $V^1$ and $V^2$ be finite dimensional real vector spaces, let $\{x_i^1\}_{i=1}^r$ and $\{y^1_i\}_{i=1}^s$ be linearly independent subsets of $V^1$, and let $\{x^2_i\}_{i=1}^r$ and $\{y^2_i\}_{i=1}^s$ be linearly independent subsets of $V^2$. If ${\sum_{i = 1}^{r} x_i^1 \otimes x_i^2} = {\sum_{i = 1}^{s} y_i^1 \otimes y_i^2}$, then $r = s$.
\end{lem}   
\begin{proof}
Since $\{x_i^1 \}_{i = 1}^r$ is independent, we can choose $ \{x_j^{1*} \}_{j = 1}^r \subseteq V^{1*}$ such that $x^{1*}_j (x_i^1) = \delta_{ij}$. By taking the mode-$1$ contraction of the tensor with respect to each of the two representations, it follows that 
\begin{equation*}
    \sum_{i = 1}^r \ x^{1*}_j(x_i^1) \ x_i^2 \ = \ \sum_{i = 1}^s x^{1*}_j(y^{1}_i) \ y^2_i  \ \text{ for } j = 1, 2, \dots, r.
\end{equation*} 
This implies $x^2_j = \sum_{i = 1}^s x^{1*}_j(y^1_i) \ y^2_i$ for $j = 1, 2, \dots, r$, so each $x^2_j \in \langle y^2_i \rangle_{i = 1}^s$. Thus, we conclude that the linear space $\langle x^2_j \rangle_{j = 1}^r$ is a subset of $\langle y^2_i \rangle_{i = 1}^s$. Similarly, $\langle y^2_i \rangle_{i = 1}^s \subseteq \langle x^2_j \rangle_{j = 1}^r $, so by the independence of each set, ${r = s}$. \end{proof}

\begin{thm}
Let $V^1$, $V^2$, and $V^3$ be finite dimensional real vector spaces. If the sets $\{x^1_1, x^1_2\}$, $\{ x^2_1, x^2_2\}$, and $\{x^3_1, x^3_2\}$ in $V^1$, $V^2$, and $V^3$, respectively, are linearly independent, then the tensor
\begin{gather*} 
  \beta \ = \ x^1_2 \otimes x^2_1 \otimes  x_1^3 \ + \ x_1^1 \otimes  x_2^2 \otimes  x_1^3 \ + \ x_1^1 \otimes  x_1^2 \otimes  x_2^3 
\end{gather*}  
is rank three.
\end{thm}
\begin{proof}
Suppose for contradiction that $\beta$ were rank less than three. That is, suppose that
\begin{gather*} 
  \beta \ = \ y^1_1 \otimes y^2_1 \otimes  y_1^3 \ + \ y_2^1 \otimes  y_2^2 \otimes  y_2^3  
\end{gather*}  
for some vectors $y^i_j$. Since $\im\Pi_i(\beta) = \langle x^i_1, x^i_2 \rangle = \langle y^i_1, y^i_2 \rangle$ for each $i$, it follows that the sets $\{y^i_1, y^i_2\}$ are independent for each $i$. Let $\{x^{1*}_1, x^{1*}_2\}$ be the dual basis of the basis $\{x^1_1, x^1_2\}$ of the subspace $\langle x^1_1, x^1_2 \rangle$ of $V^{1}$. By considering the mode-$1$ contraction of both representations of $\beta$, it follows that
\begin{align*}
\Pi_1(\beta)(x^{1*}_2) = x^{1*}_2(x^1_2)x^2_1 \otimes  x_1^3 \ + \ x^{1*}_2(x_1^1) x_2^2 \otimes  x_1^3 \ + \ x^{1*}_2(x_1^1) x_1^2 \otimes  x_2^3 \ = \  x^{1*}_2(y^1_1) y^2_1 \otimes  y_1^3 \ + \ x^{1*}_2(y_2^1)y_2^2 \otimes  y_2^3.  
\end{align*}
This implies that
\begin{align*}
     x^2_1 \otimes x^3_1 = x^{1*}_2(y^1_1) y^2_1 \otimes y^3_1 \ + \  x^{1*}_2(y_2^1) y^2_2 \otimes y^3_2. 
\end{align*}
Since the sets $\{y^i_1, y^i_2\}$ are independent, Lemma \ref{independentunique} implies that either $x^{1*}_2(y^1_1) = 0$ or $x^{1*}_2(y_2^1) = 0$. We consider the case when $x^{1*}_2(y_2^1) = 0$, and leave the remaining similar case to the reader. It follows that 
\begin{align*}
     \im \Pi_2( \ \Pi_1(\beta)(x^{1*}_2) \ ) = \im \Pi_2(x^2_1 \otimes x^3_1) = \im \Pi_2( \ x^{1*}_2(y^1_1) y^2_1 \otimes y^3_1 \ ),
\end{align*}
which implies that $\langle x^2_1 \rangle = \langle y^2_1 \rangle$. Let $k$ be a scalar such that $y^2_1 = k x^2_1$. Similarly, 
\begin{align*}
     \im \Pi_1( \ \Pi_1(\beta)(x^{1*}_2) \ ) = \im \Pi_1(x^2_1 \otimes x^3_1) = \im \Pi_1( \ x^{1*}_2(y^1_1) y^2_1 \otimes y^3_1 \ ),
\end{align*}
so $\langle x^3_1 \rangle = \langle y^3_1 \rangle$. Finally, we derive a contradiction by considering the mode-$2$ contraction of both representations of $\beta$. 
\begin{align*}
     \Pi_2(\beta)(x^{2*}_2) = \ x^1_1 \otimes x^3_1 & \ = \  x^{2*}_2(y^2_1) y^1_1 \otimes y^3_1 \  + \ x^{2*}_2(y^2_2) y^2_1 \otimes y^3_2 \\
     & \ = \ x^{2*}_2(kx^2_1) y^1_1 \otimes y^3_1 \ + \ x^{2*}_2(y^2_2) y^2_1 \otimes y^3_2 \\
    & \ = \ x^{2*}_2(y^2_2) y^2_1 \otimes y^3_2.
\end{align*}
However, this implies that $y^3_2 \in \langle x^3_1 \rangle$. This is a contradiction, since we have already shown that $y^3_1 \in \langle x^3_1 \rangle$ and the set $\{y^3_1, y^3_2\}$ is independent. We leave it to the reader to check the similar case of $x^{1*}_2(y^1_1) = 0$. 
\end{proof}

We have now shown that when $V^1$, $V^2$, and $V^3$ are finite dimensional real vector spaces and the sets $\{x^i_1, x^i_2\} \subseteq V^i$ are linearly independent, then the tensor
\begin{gather*} 
  \beta = \ x^1_2 \otimes x^2_1 \otimes  x_1^3 \ + \ x_1^1 \otimes  x_2^2 \otimes  x_1^3 \ + \ x_1^1 \otimes  x_1^2 \otimes  x_2^3 \label{tangent_form}
\end{gather*}  
is rank three. However, it is also the limit of the sequence of rank at most two tensors 
\begin{equation*}
    \tau_n = \ n \ (x_1^1 + \frac{1}{n} x_2^1) \otimes  (x_1^2 + \frac{1}{n} x_2^2) \otimes    (x_1^3 + \frac{1}{n} x_2^3) \ - \ n x_1^1 \otimes  x_1^2 \otimes x_1^3, \ n \in \N
\end{equation*}
in the norm topology. Hence, such tensors do not have optimal rank two approximations. We next characterize optimal rank two $2 \times 2 \times 2$ approximations geometrically, and show that, in fact, $\it{all}$ rank three $2 \times 2 \times 2$ tensors do not have optimal rank two approximations with respect to the Frobenius norm.

\section{A Characterization of Optimal Rank Two Approximations}

The set of simple tensors is a variety, and it is called the Segre variety. It is the image of the map 
\begin{align*}
   V^1 \times V^2 \times V^3   \ &\rightarrow \ V^1 \otimes V^2 \otimes  V^3 \\
   v^{1} \times v^2 \times v^3 \ &\mapsto \ v^1 \otimes v^2 \otimes  v^3.
\end{align*}
The Segre variety of the $n_1 \times n_2 \times n_3$ tensor space ${V^1 \otimes V^2 \otimes V^3}$ is denoted $X^{n_1 \times n_2 \times n_3}$ or simply $X$ when the dimensions are clear from context. The tangent space of the Segre variety at a point $\upsilon$ is denoted as $T_\upsilon(X)$, and is characterized in \cite{Ottaviani} as in the following theorem. 

\begin{thm}\label{tangentspace}
Let $V^1$, $V^2$, and $V^3$ be finite $n_1$, $n_2$, $n_3$-dimensional real vector spaces, respectively, and let $x^1_1 \otimes x^2_1 \otimes x^3_1$ be a rank one tensor in the Segre variety $X$. The tangent space of $X$ at $x^1_1 \otimes x^2_1 \otimes x^3_1$ is the space of all tensors in the form
\begin{gather}
   x^1_2 \otimes x^2_1 \otimes x^3_1 \ + \ x^1_1 \otimes x^2_2 \otimes x^3_1 \ + \ x^1_1 \otimes x^2_1 \otimes x^3_2  \label{tangentspace_form}
\end{gather}
for some $x^1_2 \in V^1$, $x^2_2 \in V^2$, and $x^3_2 \in V^3$.
\end{thm} 

Tensors in the form of ($\ref{tangentspace_form}$) are precisely the tensors that we previously showed were limit points of the sequence ($\ref{sequence_rank_two}$) of tensors of rank at most two. This gives us geometric insight into the tensor rank-jumping phenomenon. Secant lines of the Segre variety contain rank at most two tensors, and tangent lines are limits of secant lines. Thus, rank three tensors on lines tangent to the Segre variety are limit points of rank at most two tensors.

\begin{figure}[H]
\hspace*{1.5cm} 
    \begin{tikzpicture}
    [dot/.style={circle,inner sep=1pt,fill,label={#1},name=#1},
    extended line/.style={shorten >=-#1,shorten <=-#1},
    extended line/.default=1cm]
    \draw[thick] (0,0) .. controls (1,2) and (2,4) .. (4,2);
    \draw[thick] (4,2) .. controls (7,-1)  .. (8,-1);
    \node [dot= ] at (1.15,2) {};
    \node [dot= ] at (2.3,2.8) {};
    \node [dot= ] at (5,4.69) {};
    \node [dot= ] at (6.5,2.9) {};
    \draw[thick, blue] (-1,2.8) .. controls (2.3,2.8) .. (7,2.90302);
     \draw[thick, red] (0, 1.2)  .. controls (2.3,2.8) .. (6, 5.37);
    \put(200,-20){color{black}$X$}
    \put(-10,130){color{black}$V^1 \otimes V^2 \otimes V^3$}
    \put(180,88){color{black}$ x_1 \otimes x_2 \otimes  y_3 \ + \ x_1 \otimes  y_2 \otimes  x_3 \ + \ y_1 \otimes  x_2 \otimes  x_3$}
    \put(155,137){color{black}$n (x_1 + \frac{1}{n}y_1) \otimes (x_2 + \frac{1}{n}y_2) \otimes  (x_3 + \frac{1}{n} y_3) - nx_1 \otimes x_2 \otimes  x_3$}
    \put(192,110){color{black}$ \searrow n \to \infty$}
    \put(15,88){color{black}$ x_1 \otimes x_2 \otimes  x_3$}
    \put(30,45){color{black}$ n(x_1 + \frac{1}{n}y_1) \otimes (x_2 + \frac{1}{n}y_2) \otimes  (x_3 + \frac{1}{n} y_3)$}
\end{tikzpicture}
    \caption{}
    \label{fig:my_label}
\end{figure}

We have seen that the set of rank at most two tensors is not closed with respect to the norm topology. This implies that it is also not closed with respect to the Zariski topology. This motivates the following definition of the $2^{nd}$ secant variety of the Segre variety.

\begin{defn}
The $2^{nd}$ secant variety of the Segre variety, denoted $\sigma_2(X)$, is the Zariski closure of all the secant lines of the Segre variety $X$. 
\end{defn}

If $\upsilon$ is an optimal rank two approximation of $\tau$ with respect to an inner product norm, then $\upsilon - \tau$ must be orthogonal to the tangent space of $\sigma_2(X)$ at $\upsilon$, which we characterize in the following theorem using Theorem $\ref{tangentspace}$ and Terracini's Lemma \cite{Zak}. 

\begin{thm}\label{tangentspace_2_secant_variety}
Let $V^1$, $V^2$, and $V^3$ be finite $n_1$, $n_2$, $n_3$-dimensional real vector spaces, respectively, and let $\upsilon = x^1_1 \otimes x^2_1 \otimes x^3_1 + x^1_2 \otimes x^2_2 \otimes x^3_2$ be a rank two tensor in $V^1 \otimes V^2 \otimes V^3$. The tangent space of $\sigma_2(X)$ at  $\upsilon$ is the space of all tensors in the form
\begin{gather}\label{second_secant_form}
   x^1_3 \otimes x^2_1 \otimes x^3_1 \ + \ x^1_1 \otimes x^2_3 \otimes x^3_1 \ + \ x^1_1 \otimes x^2_1 \otimes x^3_3 \ + \ x^1_4 \otimes x^2_2 \otimes x^3_2 \ + \ x^1_2 \otimes x^2_4 \otimes x^3_2 \ + \ x^1_2 \otimes x^2_2 \otimes x^3_4  
\end{gather}
for some $x^1_3, x^1_4 \in V^1$, $x^2_3, x^2_4 \in V^2$, and $x^3_3, x^3_4 \in V^3$.
\end{thm} 

Until now, we have worked with $n_1 \times n_2 \times n_3$ tensors. The next theorem, however, is our first statement that must be restricted to $2 \times 2 \times 2$ tensors. 

\begin{thm}\label{rank_two_approx_not_generic}
Let $V^1$, $V^2$, and $V^3$ be $2$-dimensional real vector spaces, and let $\tau \in {V^1 \otimes V^2 \otimes V^3}$ be of rank greater than two. If $\upsilon \in {V^1 \otimes V^2 \otimes V^3}$ is an optimal rank two approximation of $\tau$ with respect to an inner product norm, then $\im \Pi_i(\upsilon)$ is not dimension two for some $i = 1,2,$ or $3$. 
\end{thm}
\begin{proof}
Suppose for contradiction that $\im \Pi_i(\upsilon)$ were dimension two for $i = 1,2,$ and $3$. Then, there would exist three linearly independent sets $\{ x^i_1, x^i_2  \}$ in $V^i$ for $i = 1, 2, \text{ and } 3$, such that 
\begin{equation*}
    \upsilon \ = \ x^1_1 \otimes x^2_1 \otimes x^3_1 \ + \ x^1_2 \otimes x^2_2 \otimes x^3_2.
\end{equation*}
It follows that the set $\{x^1_i \otimes x^2_j \otimes x^3_k \}_{i, j, k = 1}^{2, 2, 2}$ is a basis of ${V^1 \otimes V^2 \otimes V^3}$, so there must exist constants $c_{ijk}$ such that $\tau = \sum_{i,j,k=1}^{2,2,2} c_{ijk} \ x^1_i \otimes x^2_j \otimes x^3_k$. By multilinearity, 
\begin{align*}
    \tau &= (c_{111} x^1_1 + c_{211}x^1_2) \otimes x^2_1 \otimes x^3_1 \ + \   x^1_1 \otimes c_{121} x^2_2 \otimes x^3_1 \ + \  x^1_1 \otimes x^2_1 \otimes c_{112} x^3_2 \\
    &\hspace{20mm} + \ c_{122}x^1_1 \otimes x^2_2 \otimes x^3_2 \ + \ x^1_2 \otimes c_{212} x^2_1 \otimes x^3_2 \ + \ x^1_2 \otimes x^2_2 \otimes (c_{221}x^3_1 + c_{222}x^3_2).
\end{align*}
Hence, $\tau$ is in the form of ($\ref{second_secant_form}$), and, thus, it is in the tangent space of $\sigma_2(X)$ at $\upsilon$. This implies that $\upsilon - \tau$ is both in $T_{\upsilon}\sigma_2(X)$ and orthogonal to $T_{\upsilon}\sigma_2(X)$, which implies that $\tau = \upsilon$. However, this is a contradiction, since the rank of $\tau$ is not equal to the rank of $\upsilon$ by hypothesis.
\end{proof}

\section{P-Norms and Optimal Rank Two Approximations}

Once bases of vector spaces $V^1$, $V^2$, and $V^3$ are chosen, we can define explicit norms on the tensor space ${V^1 \otimes V^2 \otimes V^3}$. Choose basis $\{ e^i_j \}_{j=1}^{n_i}$ of $V^i$ for $i = 1$, $2$, and $3$, and denote the corresponding dual basis as $\{ e^{i*}_j \}_{j=1}^{n_i}$ also for $i =$ $1$, $2$, and $3$. Let $B$ denote the collection of these bases. For $\tau \in V^1 \otimes V^2 \otimes V^3$, we define the following class of norms for any positive integer $p$. 
\begin{equation*}
    \| \tau \|_{B,p} \ = \ \left( \sum_{i,j,k = 1,1, 1}^{n_1, n_2, n_3} | a_{ijk}|^p  \right)^{\frac{1}{p}},
\end{equation*}
where $\tau = \sum_{i,j,k} a_{ijk} \ e^1_i \otimes e^2_j \otimes e^3_k$. Similarly, for $\kappa \in V^r \otimes V^s$ and  $v^r \in V^r$ for $1 \le r, s \le 3$, we define
\begin{gather*}
     \| \kappa \|_{B, p} \ = \ \left( \sum_{i,j = 1,1}^{n_r, n_s} | a_{ij}|^p  \right)^{\frac{1}{p}}, \text{ and } \ 
    \| v^r \|_{B, p} \ = \ \left( \sum_{i = 1}^{n_r} | a_{i}|^p \right)^{\frac{1}{p}},
\end{gather*}
where ${\kappa = \sum_{i,j} a_{ij} \ e^r_i \otimes e^s_j}$ and ${v^r = \sum_{i} a_{i} \ e^r_i}$. These norms are a convenient choice for working in tensor spaces as they work well with contraction maps.  
\begin{thm}\label{property_contraction_maps_p_norms} Let $V^1$, $V^2$, and $V^3$ be finite $n_1$, $n_2$, and $n_3$-dimensional real vector spaces, and let $B$ be the collection of bases of $V^1$, $V^2$, and $V^3$ defined above. For $\tau \in V^1 \otimes V^2 \otimes V^3$,
\begin{gather*}
    \|\tau \|_{B, p}^p \ = \ \sum_{t=1}^{n_i} \ \| \Pi_i(\tau)(e^{i*}_t) \|_{B, p}^p  
\end{gather*}
for any mode-$i$ contraction $\Pi_i$ for $i = 1, 2, 3$.
\end{thm}
\begin{proof}
Without loss of generality, we prove the theorem for the the mode-$1$ contraction. Suppose $ \tau = {\sum_{s=1}^r v^1_s \otimes v^2_s \otimes v^3_s}$ for some vectors $v^i_s$, and let $a^i_{s,t}$ be scalars such that ${v^i_s = \sum_{t=1}^{n_i} a^i_{s,t} e^i_t}$ for all $i$ and $s$. It follows that 
\begin{equation*}
\sum_{s=1}^r v^1_s \otimes v^2_s \otimes v^3_s  \ = \ \sum_{i,j,k =1}^{n_1, n_2, n_3} \left( \sum_{s=1}^r a^1_{si}a^2_{sj}a^3_{sk} \right) \ e^1_i \otimes e^2_j \otimes e^3_k.
\end{equation*}
Hence, 
\begin{align*}
\| \tau \|_{B,p}^p &\ = \ \sum_{i,j,k =1}^{n_1, n_2, n_3} \left| \sum_{s=1}^r  a^1_{si}a^2_{sj}a^3_{sk} \right|^p \\
&\ = \ \sum_{i=1}^{n_1} \sum_{j,k =1}^{ n_2, n_3} \left| \sum_{s=1}^r  a^1_{si}a^2_{sj}a^3_{sk} \right|^p \ = \ \sum_{i=1}^{n_1} \| \Pi_1(\tau)(e^{1*}_i) \|_{B,p}^p.
\end{align*}  
\end{proof}

\noindent Furthermore, when $p = 2$, the norm $\| \cdot \|_{B, 2}$ is induced by the inner product
\begin{align*}
    \langle \tau, \upsilon \rangle_{B, 2} &= \sum_{i,j,k} a_{ijk} b_{ijk},
\end{align*}
where $\tau = {\sum_{ijk} a_{ijk} \ e^1_i \otimes e^2_j \otimes e^3_k}$ and $\upsilon = \sum_{ijk} b_{ijk} \ e^1_i \otimes e^2_j \otimes e^3_k$. The norm $\| \cdot \|_{B, 2}$ is called the Frobenius norm with respect to bases $B$. 

\begin{thm}\label{Frobenius_norm}
The Frobenius inner product has the following property on rank one tensors.
\begin{align*}
    &\langle  v^1_1 \otimes v^2_1 \otimes v^3_1 \mid v^1_2 \otimes v^2_2 \otimes v^3_2 \rangle_{B, 2} \ = \ \langle v^1_1 \mid v^1_2\rangle_{B, 2} \ \langle v^2_1 \mid v^2_2 \rangle_{B, 2} \ \langle v^3_1 \mid v^3_2 \rangle_{B, 2} \label{Frobenius_property}
\end{align*} 
for any rank one tensors $v^1_1 \otimes v^2_1 \otimes v^3_1, \ v^1_2 \otimes v^2_2 \otimes v^3_2 \in V^1 \otimes V^2 \otimes V^3$.
\end{thm}
\begin{proof}
Let $a^i_{t,s}$ be scalars such that ${v^i_t = \sum_{t=1}^{n_i} a^i_{t,s} e^i_s}$ for $i = 1,2,3$ and $t = 1,2$. It follows that 
\begin{align*}
v^1_1 \otimes v^2_1 \otimes v^3_1 &\ = \ \sum_{i,j,k =1}^{n_1, n_2, n_3} a^1_{1i}a^2_{1j}a^3_{1k} \ e^1_i \otimes e^2_j \otimes e^3_k, \text{ and  } \\
v^1_2 \otimes v^2_2 \otimes v^3_2 &\ = \ \sum_{i,j,k =1}^{n_1, n_2, n_3} a^1_{2i}a^2_{2j}a^3_{2k} \ e^1_i \otimes e^2_j \otimes e^3_k. 
\end{align*}
Hence,
\begin{align*}
\langle  v^1_1 \otimes v^2_1 \otimes v^3_1 \mid v^1_2 \otimes v^2_2 \otimes v^3_2 \rangle_{B, 2} &\ = \ \sum_{i,j,k =1}^{n_1, n_2, n_3} a^1_{1i}a^2_{1j}a^3_{1k} a^1_{2i} a^2_{2j} a^3_{2k} \\
& \ = \ \left(\sum_{i =1}^{n_1} a^1_{1i} a^1_{2i} \right) \left(\sum_{j =1}^{n_2} a^2_{1j} a^2_{2j} \right) \left(\sum_{k =1}^{n_3} a^3_{1k} a^3_{2k} \right)\\
& \ = \ \langle v^1_1 \mid v^1_2\rangle_{B, 2} \ \langle v^2_1 \mid v^2_2 \rangle_{B, 2} \ \langle v^3_1 \mid v^3_2 \rangle_{B, 2}. 
\end{align*}  
\end{proof}

An optimal rank $r$ approximation could be rank strictly less than $r$ by our definition. We now show that an optimal rank $r$ approximations with respect to $p$-norms must be rank $r$. This theorem is modified from \cite[Lemma 8.2]{Lim}. 

\begin{thm}\label{rank_less_than_or_equal_r_approx_ring}
Let $V^i$ be finite $n_i$-dimensional real vector spaces, and let $\tau \in V^1 \otimes V^2 \otimes V^3$ have rank greater than $r$. If $\upsilon$ is an optimal rank $r$ approximation of $\tau$ with respect to $\| \cdot \|_{B, p}$, then $\upsilon$ must be rank $r$. 
\end{thm}
\begin{proof}
Suppose for contradiction that there existed an optimal rank $r$ approximation $\upsilon$ with rank strictly less than $r$. Let $a_{ijk}$ and $b_{ijk}$ be scalars such that $\tau = \sum_{ijk} a_{ijk} \ e^1_i \otimes e^2_j \otimes e^3_k$ and $\upsilon = \sum_{ijk} b_{ijk} \ e^1_i \otimes e^2_j \otimes e^3_k$. Since $\tau$ and $\upsilon$ have different ranks, there must exists some triple $(\alpha,\beta,\gamma)$ such that $a_{\alpha \beta \gamma} \neq b_{\alpha \beta \gamma}$. The tensor $\upsilon' = \upsilon + (a_{\alpha \beta \gamma} - b_{\alpha \beta \gamma}) e^1_{\alpha} \otimes e^2_{\beta} \otimes e^3_{\gamma}$ is rank less than or equal to $r$ by construction. It follows that
\begin{align*}
    \| \tau - \upsilon' \|_{B,p} &= \| \tau - \left(\upsilon + (a_{\alpha \beta \gamma} - b_{\alpha \beta \gamma}) e^1_{\alpha} \otimes e^2_{\beta} \otimes e^3_{\gamma} \right) \|_{B,p} \\
    &= \left( \sum_{ijk \ne \alpha \beta \gamma} |a_{ijk} - b_{ijk}|^p \right)^{\frac{1}{p}} \\
    &< \left( \sum_{ijk} |a_{ijk} - b_{ijk}|^p \right)^{\frac{1}{p}} =  \| \tau - \upsilon \|_{B,p},
\end{align*}
which contradicts that $\upsilon$ is an optimal rank $r$ approximation of $\tau$. 
\end{proof}

If $\upsilon$ is an optimal rank two approximation of $\tau$ with respect to the Frobenius norm, then the contractions of $\upsilon$ must be related to the contractions of $\tau$ in the following way.

\begin{thm}\label{projections_rank_two_approx}
Let $V^1$, $V^2$, and $V^3$ be $n_1$, $n_2$, and $n_3$-dimensional real vector spaces, respectively. Let $B$ denote the collection of bases $\{ e^i_j \}_{j=1}^{n_i}$ of $V^i$ for $i =$ $1$, $2$, and $3$. Furthermore, denote the corresponding dual bases as $\{ e^{i*}_j \}_{j=1}^{n_i}$ for $i =$ $1$, $2$, and $3$. Let $\tau \in {V^1 \otimes V^2 \otimes V^3}$ be of rank greater than or equal to two, and let $\upsilon \in {V^1 \otimes V^2 \otimes V^3}$ be an optimal rank two approximation of $\tau$ with respect to the Frobenius norm $\| \cdot \|_{B, 2}$. Let $P_{\im \Pi_i(\upsilon)}$ denote the projection onto the image of the mode-$i$ contraction of $\upsilon$. It follows that
\begin{equation*}
   P_{\im \Pi_i(\upsilon)}(\Pi_i(\tau)(e^{i*}_j)) = \Pi_i(\upsilon)(e^{i*}_j) \text{ for any } i, j.
\end{equation*}
\end{thm}
\begin{proof}
As $\upsilon$ is rank two, there must exist vectors $x_j^i$ such that
\begin{align*}
    \upsilon &= \ x^1_1 \otimes x^2_1 \otimes x^3_1 \ + \ x^1_2 \otimes x^2_2 \otimes x^3_2.
\end{align*}
Without loss of generality, we prove the theorem for $i$ and $j$ both equal to $1$. First, note that the image of $\Pi_1(\upsilon)$ contains vectors in the form $ax^2_1 \otimes x^3_1 + bx^2_2 \otimes x^3_2$ for some constants $a$ and $b$. If the set $\{x^1_1, x^1_2\}$ is independent, it can be extended to a basis $\{x^1_j\}_{j=1}^{n_1}$ with dual basis $\{x^{1*}_j\}_{j=1}^{n_1}$. It then follows that the image of $\Pi_1(\upsilon)$ is the span of $\Pi_1(\upsilon)(x^{1*}_1)$ $= $ $x^2_1 \otimes x^3_1$ and $\Pi_1(\upsilon)(x^{1*}_2)$ $= $ $x^2_2 \otimes x^3_2$, so every element in the image of $\Pi_1(\upsilon)$ can indeed be written as  $ax^2_1 \otimes x^3_1 + bx^2_2 \otimes x^3_2$ for some constants $a$ and $b$. On the other hand, if $x^1_2 = kx^1_1$ for some constant $k$, then the set $\{x^1_1\}$ can similarly be extended to a basis $\{x^1_1, y^1_j\}_{j=2}^{n_1}$ of $V^1$ with the corresponding dual basis $\{x^{1*}_1, y^{1*}_j\}_{j=2}^{n_1}$. It follows that the image of $\Pi_1(\upsilon)$ is the span of $\Pi_1(\upsilon)(x^{1*}_1)$ $= $ $x^2_1 \otimes x^3_1$  $+$ $kx^2_2 \otimes x^3_2$. Hence, in this case, it is also true that every element in the image of $\Pi_1(\upsilon)$ can be written as $ax^2_1 \otimes x^3_1 + bx^2_2 \otimes x^3_2$ for some constants $a$ and $b$. 

Suppose for contradiction that $P_{\im \Pi_1(\upsilon)}(\Pi_1(\tau)(e^{1*}_1)) \neq$ $\Pi_1(\upsilon)(e^{1*}_1)$. Let $\alpha$ and $\beta$ be scalars such that
\begin{gather*}
    P_{\im \Pi_1(\upsilon)}(\Pi_1(\tau)(e^{1*}_1)) \ = \ \alpha \ x^2_1 \otimes x^3_1 \ + \ \beta \ x^2_2 \otimes x^3_2.
\end{gather*}
Furthermore, let $a_i$ and $b_i$ be scalars such that
\begin{gather*}
   \Pi_1(\upsilon)(e^{1*}_i) \ = \ a_i \ x^2_1 \otimes x^3_1 \ + \ b_i \ x^2_2 \otimes x^3_2
\end{gather*}
for $i = 2, \ldots, n_1$. Define $\upsilon' \in {V^1 \otimes V^2 \otimes V^3}$ as the unique tensor with the following mode-$1$ contraction:
\begin{align*}
\Pi_1(\upsilon')(e^{1*}_1) \ &= \ \alpha \ x^2_1 \otimes x^3_1 \ + \ \beta \ x^2_2 \otimes x^3_2 \ = \ P_{\im \Pi_1(\upsilon)}(\Pi_1(\tau)(e^{1*}_1)), \\
\Pi_1(\upsilon')(e^{1*}_i) \ &= \ a_i \ x^2_1 \otimes x^3_1 \ + \ b_i \ x^2_2 \otimes x^3_2 \ = \ \Pi_1(\upsilon)(e^{1*}_i) \ \text{ for } i = 2, \ldots, n_1.
\end{align*}
It follows that 
\begin{align*}
    \upsilon' &\ = \ (\alpha e^1_1 + \sum_{j=2}^{n_1} a_j e^1_j) \otimes x^2_1 \otimes x^3_1 \ + \ (\beta e^1_1 + \sum_{j=2}^{n_1} b_j e^1_j) \otimes x^2_2 + \otimes x^3_2,
\end{align*}
so $\upsilon'$ is rank $\le 2$. Furthermore, 
\begin{align}
    \| \tau - \upsilon' \|_{B,2}^2 &\ = \ \sum_{i=1}^{n_1} \| \Pi_1(\tau)(e^{1*}_i)  -  \Pi_1(\upsilon')(e^{1*}_i) \|_{B,2}^2 \label{proof_line_1}\\
     &\ = \ \| \Pi_1(\tau)(e^{1*}_1) - \Pi_1(\upsilon')(e^{1*}_1) \|_{B,2}^2 \ + \ \sum_{i=2}^{n_1} \| \Pi_1(\tau)(e^{1*}_i) - \Pi_1(\upsilon)(e^{1*}_i) \|_{B,2}^2 \label{proof_line_2}\\
     &\ < \ \| \Pi_1(\tau)(e^{1*}_1) - \Pi_1(\upsilon)(e^{1*}_1) \|_{B,2}^2 \ + \ \sum_{i=2}^{n_1} \| \Pi_1(\tau)(e^{1*}_i) - \Pi_1(\upsilon)(e^{1*}_i) \|_{B,2}^2 \label{proof_line_3}\\
     & \ = \ \| \tau - \upsilon \|_{B,2}^2. \nonumber
\end{align}

Equation ($\ref{proof_line_1}$) follows from Theorem $\ref{property_contraction_maps_p_norms}$, and equation ($\ref{proof_line_2}$) follows from the fact that $\Pi_1(\upsilon')(e^{1*}_i) = \Pi_1(\upsilon)(e^{1*}_i)$ for $i = 2, 3, \ldots, n_1$. Equation ($\ref{proof_line_3}$) follows from our hypothesis that $\alpha x^2_1 \otimes x^3_1 + \beta x^2_2 \otimes x^3_2$ is a better approximation of $\Pi_1(\tau)(e^{1*}_1)$ with respect to $\| \cdot \|_{B,2}$ than $\Pi_1(\upsilon)(e_1^{1*})$. Hence, $\| \tau - \upsilon' \|_{B,2} < \| \tau - \upsilon \|_{B,2}$, contradicting that $\upsilon$ is a best rank two approximation of $\tau$.
\end{proof}

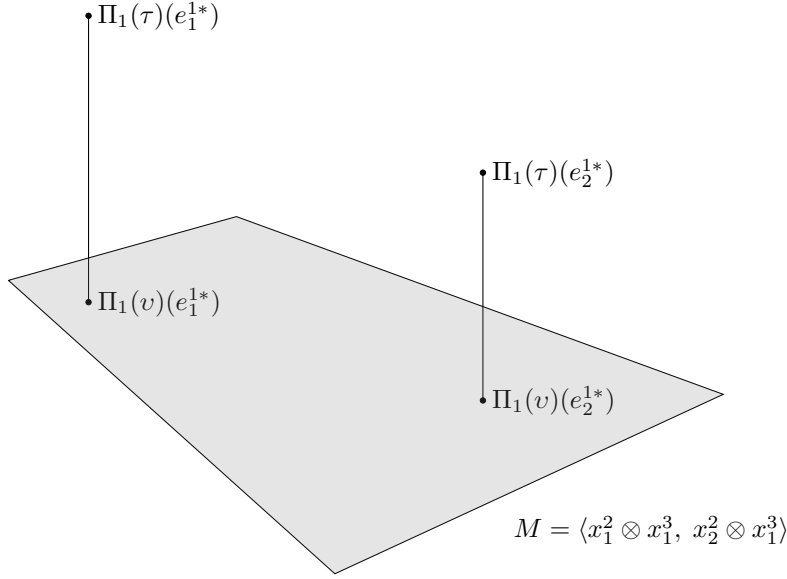
\begin{figure}
    \centering
\tdplotsetmaincoords{60}{125}
\begin{tikzpicture}[tdplot_main_coords]
\coordinate (c) at (0, 0.-1, 2.7);
\coordinate (a) at (0, 0.-1, -1.7){};
\filldraw [black] (a) circle (1pt);
\filldraw [black] (c) circle (1pt);
\node [anchor = west] at (a) {$\Pi_1(\upsilon)(e_1^{1*})$};
\node [anchor = west] at (c) {$\Pi_1(\tau)(e_1^{1*})$};
\draw (a) -- (c);

\coordinate (f) at (-2, 4, 1);
\coordinate (d) at (-2, 4, -2.5){};
\filldraw [black] (d) circle (1pt);
\filldraw [black] (f) circle (1pt);
\node [anchor = west] at (d) {$\Pi_1(\upsilon)(e_2^{1*})$};
\node [anchor = west] at (f) {$\Pi_1(\tau)(e_2^{1*})$};
\draw (d) -- (f);

\coordinate (x) at (-1, -3, -2.5){};
\coordinate (y) at (-2, 0, -1){};
\coordinate (z) at (10, 10, 2.5){};
\coordinate (w) at (1, 10, 1){};
\draw [black] (x) -- (y) -- (w) -- (z) -- (x);

   \begin{scope}[canvas is yz plane at x=0]
     \fill[gray,fill opacity=0.2,name path global=plane] (x) -- (y) -- (w) -- (z) -- cycle; 
   \end{scope}

\coordinate (h) at (2.5, 7.5,-1.2){};
\node [anchor = west] at (h) {$M = \langle x^2_1 \otimes x^3_1, \ x^2_2 \otimes x^3_1 \rangle $};
\coordinate (r) at (2.5, 7.5,-2){};
\node [anchor = west] at (r) {};
\end{tikzpicture}
\parbox{5in}{\caption{In this figure, we consider the $2 \times 2 \times 2$ case. Let $\upsilon = {x^1_1 \otimes x^2_1 \otimes x^3_1} + {x^1_2 \otimes x^2_2 \otimes x^3_1}$ be an optimal rank two approximation of rank three tensor ${\tau \in V^1 \otimes V^2 \otimes V^3}$. In this case, $\im \Pi_1(\upsilon)$ is a plane in $V^2 \otimes V^3$ that contains all the rank one tensors in the span of $x^2_1 \otimes x^3_1$ and $x^2_2 \otimes x^3_1$. This plane is denoted as $M$ in the figure. Theorem \ref{projections_rank_two_approx} shows that the projection of $\Pi_1(\tau)(e^{1*}_i)$ onto the plane $M$ must be $\Pi_1(\upsilon)(e^{1*}_i)$ for $i = 1, 2$.}\label{projection_fig}}
\end{figure}

\section{Optimal Rank Two Approximations of Rank Three $2 \times 2 \times 2$ Tensors Do Not Exist With Respect to the Frobenius Norm}

In this section, we show that rank three $2 \times 2 \times 2$ tensors over $\R$ do not have optimal rank two approximations with respect to the Frobenius norm. Given any rank two tensor $\upsilon$ and rank three tensor $\tau$, we construct a rank two tensor $\upsilon'$ that is a better approximation of $\tau$ than $\upsilon$ with respect to the Frobenius norm. Let $V^i$ be two-dimensional real vector spaces for $i = 1, 2, 3$. As in the previous section, let $B$ denote a collection of bases $\{ e^i_j \}_{j=1}^{2}$ of $V^i$ for $i =$ $1$, $2$, $3$. Furthermore, denote the corresponding dual bases as $\{ e^{i*}_j \}_{j=1}^2$ for $i =$ $1$, $2$, $3$.

\begin{thm}\label{best_approx_dont_exists}
If $\tau \in V^1 \otimes V^2 \otimes V^3$ is rank three, then there does not exist an optimal rank two approximation of $\tau$ with respect to $\| \cdot \|_{B, 2}$. 
\end{thm}
\begin{proof}
Suppose for contradiction that there existed a tensor $\upsilon$ that was an optimal rank two approximation of $\tau$. By Theorem $\ref{rank_two_approx_not_generic}$, we may assume 
\begin{gather*}
    \upsilon \ = \ x^1_1 \otimes x^2_1 \otimes x^3_1 \ + \ x^1_2 \otimes x^2_2 \otimes x^3_1
\end{gather*} for some vectors $x^i_j$.
Since $\{e^1_1, e^1_2\}$ is a basis of $V^1$, there exists scalars $a$, $b$, $c$, and $d$ such that
\begin{gather*}
    x^1_1 = ae^1_1 + ce^1_2 \ \text{ and } \ x^1_2 = be^1_1 + de^1_2.
\end{gather*}
It follows that 
\begin{gather*}
    \upsilon \ = \ e^1_1 \otimes (ax^2_1 + bx^2_2) \otimes x^3_1 \ + \ e^1_2 \otimes (cx^2_1 + dx^2_2) \otimes x^3_1.
\end{gather*}
If the set $\{x^1_1, x^1_2\}$ were linearly dependent, then $\upsilon$ would be rank one by multilinearity, which would contradict Theorem $\ref{rank_less_than_or_equal_r_approx_ring}$. Hence, $\{x^1_1, x^1_2\}$ is independent, and is thus a basis of $V^1$. Let $\{x^{1*}_1, x^{1*}_2\}$ denote its dual basis. The image of $\Pi_1(\upsilon)$ is the span of $\Pi_1(\upsilon)(x^{1*}_1)$ $= $ $x^2_1 \otimes x^3_1$ and $\Pi_1(\upsilon)(x^{1*}_2)$ $= $ $x^2_2 \otimes x^3_1$. These two tensors are linearly independent in $V^2 \otimes V^3$ since the set $\{x^2_1, x^2_2\}$ is independent, which also follows from the fact that $\upsilon$ is rank two. Hence, the image of $\Pi_1(\upsilon)$ is the plane spanned by the tensors ${x^2_1 \otimes x^3_1}$ and ${x^2_2 \otimes x^3_1}$. Let $M$ denote this plane and let $P_M$ denote the projection onto this plane. By Theorem $\ref{projections_rank_two_approx}$, 
\begin{align*}
    P_M(\Pi_1(\tau)(e^{1*}_1)) &= \Pi_1(\upsilon)(e^{1*}_1) = (ax^2_1 + bx^2_2) \otimes x^3_1, \text{ and } \\
    P_M(\Pi_1(\tau)(e^{1*}_2)) &= \Pi_1(\upsilon)(e^{1*}_2) = (cx^2_1 + dx^2_2) \otimes x^3_1.
\end{align*}
\noindent Let $x^3_2 \in V^3$ be a vector such that the set $\{x^3_1, x^3_2\}$ is orthogonal with respect to ${\langle \cdot  \mid  \cdot \rangle_{B,2}}$. The set $\{ {x^2_1 \otimes x^3_2}, \ {x^2_2 \otimes x^3_2} \}$ spans the orthogonal complement of the plane $M$ in ${V^2 \otimes V^3}$ since
\begin{align*}
    \langle x^2_i \otimes x^3_2 \mid x^2_j \otimes x^3_1 \rangle_{B,2} =  \langle x^2_i \mid x^2_j \rangle_{B,2} \ \langle x^3_2 \mid x^3_1 \rangle_{B,2} = 0 \text{ for } i, j \in \{1,2\},
\end{align*}
by property ($\ref{Frobenius_norm}$). Thus, there must exist some constants $r$, $s$, $p$, and $q$ such that
\begin{align*}
    \Pi_1(\tau)(e^{1*}_1) &= (ax^2_1 + bx^2_2) \otimes x^3_1 + (rx^2_1 + sx^2_2) \otimes x^3_2, \text{ and } \\
    \Pi_1(\tau)(e^{1*}_2) &= (cx^2_1 + dx^2_2) \otimes x^3_1 + (px^2_1 + qx^2_2) \otimes x^3_2.
\end{align*}
It follows that $\tau$ can be written in the form 
\begin{align*}
    \tau     &\ = \ e^1_1 \otimes (ax^2_1 + bx^2_2) \otimes x^3_1 \ + \ e^1_1 \otimes (rx^2_1 + sx^2_2) \otimes x^3_2 \\
    &\hspace{10mm} \ + \ e^1_2 \otimes (cx^2_1 + dx^2_2) \otimes x^3_1 \ + \ e^1_2 \otimes (px^2_1 + qx^2_2) \otimes x^3_2. 
\end{align*}

\noindent If we can show that 
\begin{align}
      &\langle \ (cx^2_1 + dx^2_2) \mid  (rx^2_1 + sx^2_2) \ \rangle_{B,2} = 0, \label{equality_3} \\
      &\langle \ (ax^2_1 + bx^2_2) \mid  (px^2_1 + qx^2_2) \ \rangle_{B,2} = 0, \text{ and }  \label{equality_2}\\
      &\langle \ (ax^2_1 + bx^2_2) \mid  (rx^2_1 + sx^2_2) \ \rangle_{B,2} = 0, \label{equality_1} 
\end{align}
then this would imply that 
\begin{equation*}
    (px^2_1 + qx^2_2) = k_1(rx^2_1 + sx^2_2) \text{ and } (cx^2_1 + dx^2_2) = k_2(ax^2_1 + bx^2_2) 
\end{equation*}
\noindent for some constants $k_1$, $k_2$, since these vectors are in a two-dimensional space. It then follows by the multilinearity of the tensor product that
\begin{align*}
    \tau &\ = \ e^1_1 \otimes (ax^2_1 + bx^2_2) \otimes x^3_1 \ + \ e^1_1 \otimes (rx^2_1 + sx^2_2) \otimes x^3_2 \\
    &\hspace{15mm} + \ e^1_2 \otimes k_2(ax^2_1 + bx^2_2)\otimes x^3_1 \ + \ e^1_2 \otimes k_1(rx^2_1 + sx^2_2) \otimes x^3_2 \\
    &\ = \ (e^1_1 + k_2e^1_2) \otimes (ax^2_1 + bx^2_2) \otimes x^3_1 \ + \ (e^1_1 + k_1e^1_2) \otimes (rx^2_1 + sx^2_2) \otimes x^3_2,
\end{align*}
contradicting that $\tau$ is rank three. 

We first prove equality ($\ref{equality_3}$) by considering the tensor
\begin{align*}
    &\upsilon_1(\epsilon) = e^1_2 \otimes (cx^2_1 + dx^2_2) \otimes x^3_1 \ +  \ e^1_1 \otimes (ax^2_1 + bx^2_2) \otimes x^3_1 \ + \ e^1_1 \otimes (cx^2_1 + dx^2_2) \otimes \epsilon x^3_2.
\end{align*}

\noindent Suppose for contradiction that ($\ref{equality_3}$) were not true. It would then follow that 
\begin{equation*}
    \langle \ (rx^2_1 + sx^2_2) \otimes x^3_2 \mid (cx^2_1 + dx^2_2) \otimes x^3_2 \ \rangle_{B, 2}  = \langle  (rx^2_1 + sx^2_2) \mid (cx^2_1 + dx^2_2) \rangle_{B, 2} \ \|x^3_2\|_{B,2}^2 \neq 0.
    \end{equation*}
\noindent We can always choose a real $\epsilon$ small enough in absolute value such that 
\begin{gather*}
   B_1(\epsilon) =  -2\epsilon \ \langle \ (rx^2_1 + sx^2_2) \otimes x^3_2 \mid (cx^2_1 + dx^2_2) \otimes x^3_2 \ \rangle_{B, 2} \  +  \ \epsilon^2 \ \| (cx^2_1 + dx^2_2) \otimes x^3_2 \|^2_{B, 2} 
\end{gather*}
is negative. For example, if $\langle \ (rx^2_1 + sx^2_2) \otimes x^3_2 \mid (cx^2_1 + dx^2_2) \otimes x^3_2 \ \rangle_{B, 2}  < 0$, an $\epsilon < 0$ small enough in absolute value would result in a $B_1(\epsilon)$ negative. Suppose such an $\epsilon$ is chosen. Observe that
\begin{align*}
    \Pi_1(\upsilon_1(\epsilon))(e^{1*}_1) &= (ax^2_1 + bx^2_2) \otimes x^3_1 \ + \ \epsilon (cx^2_1 + dx^2_2) \otimes x^3_2, \text{ and } \\
    \Pi_1(\upsilon_1(\epsilon))(e^{1*}_2) &= (cx^2_1 + dx^2_2) \otimes x^3_1 =  \Pi_1(\upsilon)(e^{1*}_2). \nonumber
\end{align*}
\noindent It thus follows that
\begin{align}
    \| \tau - \upsilon_1(\epsilon) \|_{B, 2}^2 & \ = \ \| \Pi_1(\tau)(e^{1*}_1) - \Pi_1(\upsilon_1(\epsilon))(e^{1*}_1) \|_{B, 2}^2 \ + \ \| \Pi_1(\tau)(e^{1*}_2) - \Pi_1(\upsilon_1(\epsilon))(e^{1*}_2) \|^2_{B, 2}  \label{line_1}\\ 
    &\ = \ \| \Pi_1(\tau)(e^{1*}_1) - \Pi_1(\upsilon_1(\epsilon))(e^{1*}_1) \|^2_{B, 2}  \ + \ \| \Pi_1(\tau)(e^{1*}_2) - \Pi_1(\upsilon)(e^{1*}_2) \|^2_{B, 2}.  \label{line_2}
\end{align}
Equation ($\ref{line_1}$) follows from Theorem $\ref{property_contraction_maps_p_norms}$, and equation ($\ref{line_2}$) follows from the fact that $\Pi_1(\upsilon)(e^{1*}_2) = \Pi_1(\upsilon_1(\epsilon))(e^{1*}_2)$. By the multilinearity of the tensor product and the multilinearity of the inner product, we conclude that 
\begin{align*}
    \| \tau - \upsilon_1(\epsilon) \|_{B, 2}^2 & \ = \ \langle \ \Pi_1(\tau)(e^{1*}_1) - \Pi_1(\upsilon_1(\epsilon))(e^{1*}_1) \ \mid \ \Pi_1(\tau)(e^{1*}_1) - \Pi_1(\upsilon_1(\epsilon))(e^{1*}_1) \ \rangle_{B, 2}  \\
    &\hspace{30mm} + \ \| \Pi_1(\tau)(e^{1*}_2) - \Pi_1(\upsilon)(e^{1*}_2) \|^2_{B, 2}  \\
    & \ = \ \langle \ (rx^2_1 + sx^2_2) \otimes x^3_2 - \epsilon (cx^2_1 + dx^2_2) \otimes x^3_2 \ \mid \ (rx^2_1 + sx^2_2) \otimes x^3_2 - \epsilon (cx^2_1 + dx^2_2) \otimes x^3_2 \ \rangle_{B, 2}  \\
    &\hspace{30mm} + \ \| \Pi_1(\tau)(e^{1*}_2) - \Pi_1(\upsilon)(e^{1*}_2) \|^2_{B, 2}  \\
    & \ = \ \| (rx^2_1 + sx^2_2) \otimes x^3_2 \|^2_{B, 2}  \  - 2 \epsilon \ \langle \ (rx^2_1 + sx^2_2) \otimes x^3_2 \ \mid \ (cx^2_1 + dx^2_2) \otimes x^3_2 \ \rangle_{B, 2}     \\
    &\hspace{30mm}+ \ \epsilon^2 \ \| (cx^2_1 + dx^2_2) \otimes x^3_2 \|^2_{B, 2} \ + \ \| \Pi_1(\tau)(e^{1*}_2) - \Pi_1(\upsilon)(e^{1*}_2) \|^2_{B, 2}  \\
     & \ = \ \| (rx^2_1 + sx^2_2) \otimes x^3_2 \|^2_{B, 2} \  + \ B_1(\epsilon) \ + \ \| \Pi_1(\tau)(e^{1*}_2) - \Pi_1(\upsilon)(e^{1*}_2) \|^2_{B, 2}.
\end{align*}
Since $\epsilon$ was chosen specifically to make $B_1(\epsilon)$ negative, we can conclude that
\begin{align*}
 \| \tau - \upsilon_1(\epsilon) \|_{B, 2}^2 & \ = \ \| (rx^2_1 + sx^2_2) \otimes x^3_2 \|^2_{B, 2} \  + \ B_1(\epsilon) \ + \ \| \Pi_1(\tau)(e^{1*}_2) - \Pi_1(\upsilon)(e^{1*}_2) \|^2_{B, 2}\\
    & \ < \  \| (rx^2_1 + sx^2_2) \otimes x^3_2 \|^2_{B, 2}  \ + \ \| \Pi_1(\tau)(e^{1*}_2) - \Pi_1(\upsilon)(e^{1*}_2) \|^2_{B, 2}  \\
    & \ = \ \| \Pi_1(\tau)(e^{1*}_1) - \Pi_1(\upsilon)(e^{1*}_1) \|^2_{B, 2}  \ + \ \| \Pi_1(\tau)(e^{1*}_2) - \Pi_1(\upsilon)(e^{1*}_2) \|^2_{B, 2}  \\
    & \ = \ \| \tau - \upsilon \|^2_{B, 2}.
\end{align*}
Hence, $\upsilon_1(\epsilon)$ is a better approximation of $\tau$ than $\upsilon$. The tensor $\upsilon_1(\epsilon)$ is also in the tangent space of the Segre variety at $e^1_1 \otimes (cx^2_1 + dx^2_2) \otimes x^3_1$. Hence, there exists a sequence of rank two tensors that converges to $\upsilon_1(\epsilon)$. Thus, there must be some rank two tensor in the sequence that is better approximation to $\tau$ than $\upsilon$, contradicting that $\upsilon$ is an optimal rank two approximation. 

Equality ($\ref{equality_2}$) can be proven in the same way by considering the tensor 
\begin{align*}
     &\upsilon_2(\epsilon) \ = \ e^1_1 \otimes (ax^2_1 + bx^2_2) \otimes x^3_1 \ + \ e^1_2 \otimes (cx^2_1 + dx^2_2) \otimes x^3_1 \ + \ e^1_2 \otimes (ax^2_1 + bx^2_2) \otimes \epsilon x^3_2,
\end{align*}
which is in the tangent space of the Segre variety at ${e^1_2 \otimes (ax^2_1 + bx^2_2) \otimes x^3_1}$. 

It thus remains to show equality ($\ref{equality_1}$), which we prove by considering the rank two tensor
\begin{align*}
    \upsilon_3(\epsilon) \ = \ e^1_1 \otimes (ax^2_1 + bx^2_2) \otimes (x^3_1 + \epsilon( x^3_1 + x^3_2)) \ + \ e^1_2 \otimes (cx^2_1 + dx^2_2) \otimes x^3_1.
\end{align*}
It follows that the mode-$1$ contractions of $\upsilon_3(\epsilon)$ are as follows. 
\begin{align*}
    \Pi_1(\upsilon_3(\epsilon))(e^{1*}_1) & \ = \ (ax^2_1 + bx^2_2) \otimes x^3_1 \ + \ \epsilon (ax^2_1 + bx^2_2) \otimes (x^3_1 + x^3_2), \text{ and} \\
    \Pi_1(\upsilon_3(\epsilon))(e^{1*}_2) & \ = \ (cx^2_1 + dx^2_2) \otimes x^3_1 \ = \ \Pi_1(\upsilon)(e^{1*}_2).
\end{align*}
Suppose for contradiction that ($\ref{equality_1}$) were nonzero. It would then follow that 
\begin{align*}
    \langle \ (rx^2_1 + sx^2_2) \otimes x^3_2 \mid  (ax^2_1 + bx^2_2) \otimes (x^3_1 + x^3_2) \ \rangle_{B,2} =  \langle \ (rx^2_1 + sx^2_2) \mid (ax^2_1 + bx^2_2)  \ \rangle_{B,2} \ \|x^3_2\|^2_{B,2} \neq 0.
\end{align*}
\noindent We could then choose an $\epsilon$ small enough in absolute value such that $D(\epsilon) < 0$, where
\begin{align*}
    D(\epsilon) &= \ - 2 \epsilon \ \langle (rx^2_1 + sx^2_2) \otimes x^3_2 \mid (ax^2_1 + bx^2_2) \otimes (x^3_1 + x^3_2) \rangle_{B,2}  \ + \ \epsilon^2 \ \| (ax^2_1 + bx^2_2) \otimes (x^3_1 + x^3_2) \|_{B,2}^2.
\end{align*}
\noindent For example, if $\langle \ (rx^2_1 + sx^2_2) \otimes x^3_2 \mid  (ax^2_1 + bx^2_2) \otimes (x^3_1 + x^3_2) \ \rangle_{B,2} < 0$, an $\epsilon < 0$ small enough in absolute value will yield a negative $D(\epsilon)$. Suppose such an $\epsilon$ is chosen. Then,
\begin{align*}
\|\tau - \upsilon_3(\epsilon) \|_{B,2}^2 & \ = \  \|\Pi_1(\tau)(e^{1*}_1) - \Pi_1(\upsilon_3(\epsilon))(e^{1*}_1) \|^2_{B,2} \ + \ \|\Pi_1(\tau)(e^{1*}_2) - \Pi_1(\upsilon_3(\epsilon))(e^{1*}_2) \|^2_{B,2} \\
&\ = \  \|\Pi_1(\tau)(e^{1*}_1) - \Pi_1(\upsilon_3(\epsilon))(e^{1*}_1) \|^2_{B,2} \ + \ \|\Pi_1(\tau)(e^{1*}_2) - \Pi_1(\upsilon)(e^{1*}_2) \|^2_{B,2}. 
\end{align*}
Observe that
\begin{gather*}
     \|\Pi_1(\tau)(e^{1*}_1) - \Pi_1(\upsilon_3(\epsilon))(e^{1*}_1) \|^2_{B,2} \ = \ \|(rx^2_1 + sx^2_2) \otimes x^3_2 - \epsilon (ax^2_1 + bx^2_2) \otimes (x^3_1 + x^3_2) \|_{B,2}^2.
\end{gather*}
Writing this as an inner product, we see that
\begin{align*}
\|\tau - \upsilon_3(\epsilon) \|_{B,2}^2 & \ = \  \| (rx^2_1 + sx^2_2) \otimes x^3_2 \|_{B,2}^2 \ - \ 2 \epsilon \ \langle (rx^2_1 + sx^2_2) \otimes x^3_2 \mid (ax^2_1 + bx^2_2) \otimes (x^3_1 + x^3_2) \rangle_{B,2}   \\
&\hspace{15mm} + \ \epsilon^2 \ \| (ax^2_1 + bx^2_2) \otimes (x^3_1 + x^3_2) \|_{B,2}^2 \ + \ \|\Pi_1(\tau)(e^{1*}_2) - \Pi_1(\upsilon)(e^{1*}_2) \|^2_{B,2} \\
& \ = \  \| (rx^2_1 + sx^2_2) \otimes x^3_2 \|_{B,2}^2 \ + \ D(\epsilon) \ + \ \|\Pi_1(\tau)(e^{1*}_2) - \Pi_1(\upsilon)(e^{1*}_2) \|^2_{B,2}.
\end{align*}
Finally, since $\epsilon$ was chosen so that $D(\epsilon)$ would be negative, it follows that
\begin{align*}
\|\tau - \upsilon_3(\epsilon) \|_{B,2}^2 & \ < \  \| (rx^2_1 + sx^2_2) \otimes x^3_2 \|_{B,2}^2 \ + \ \|\Pi_1(\tau)(e^{1*}_2) - \Pi_1(\upsilon)(e^{1*}_2) \|^2_{B,2} \\
& \ = \ \|\Pi_1(\tau)(e^{1*}_1) - \Pi_1(\upsilon)(e^{1*}_1) \|^2_{B,2} \ + \ \|\Pi_1(\tau)(e^{1*}_2) - \Pi_1(\upsilon)(e^{1*}_2) \|^2_{B,2} \\
&\ = \ \| \tau - \upsilon \|^2_{B,2}.
\end{align*}
\noindent Hence, the rank two tensor $\upsilon_3(\epsilon)$ is a better approximation of $\tau$ than $\upsilon$, contradicting that $\upsilon$ is an optimal rank two approximation.  \end{proof}

We have thus shown that rank three $2 \times 2 \times 2$ real tensors have no optimal rank two approximations with respect to the Frobenius norm. This implies that the nearest point of a rank three $2 \times 2 \times 2$ tensor $\tau$ to the second secant variety $\sigma_2(X)$ of the Segre variety $X$ with respect to the Frobenius norm must be rank three. In fact, our proof above demonstrates that the nearest point of $\tau$ to $\sigma_2(X)$ is in fact on the tangential variety of the Segre variety, which is the set of all tensors contained in the tangent space of the Segre variety at some point $u \in X$. This is a variety as is it the daul projective variety of $X$, and is denoted $TX$. That is, 
\begin{equation*}
    TX = \{\tau \in V^1 \otimes V^2 \otimes V^3 \mid \tau \in T_u(X) \text{ for some } u \in X \}.
\end{equation*}
From Theorem (\ref{tangentspace}), we know that 
\begin{align*}
    TX &= \{\tau \in V^1 \otimes V^2 \otimes V^3 \ \mid \ \tau \ = \ x^1_2 \otimes x^2_1 \otimes x^3_1 \ + \  x^1_1 \otimes x^2_2 \otimes x^3_1 \ + \ x^1_1 \otimes x^2_1 \otimes x^3_2 \\
    & \hspace{90mm} \text{ for some } x^i_j \in V^i \}.
\end{align*}
There is an open, dense subset of $TX$ of rank three tensors. However, there are rank two elements in $TX$, and the fact that these rank two tensors can never be the the nearest point on $TX$ of any rank three tensor implies there is an interesting curvature of $TX$ at these points. We now consider an example of such a rank two tensor in $TX$. Let $\{x^i_1, x^i_2\}$ be independent vectors in $V^i$ for $i = 1,2,3$. The tensor
\begin{gather*}
    \nu  \ = \ x^1_1 \otimes x^2_1 \otimes x^3_1 \ + \ x^1_2 \otimes x^2_2 \otimes x^3_1
\end{gather*}
is a rank two tensor in $TX$. For every $\epsilon \ne 0$, the tensor
\begin{equation*}
\tau(\epsilon) \ = \ x^1_1 \otimes x^2_1 \otimes x^3_1 \ + \ x^1_2 \otimes x^2_2 \otimes x^3_1 \ + \ \epsilon(x^1_1 + x^1_2) \otimes (x^2_1 + x^2_2) \otimes x^3_2    
\end{equation*}
is rank three. Clearly, ${\lim\limits_{\epsilon \to 0} \tau(\epsilon) = \nu}$. However, the nearest point to $\tau(\epsilon)$ on $TX$ with respect to the Frobenius norm is never $\nu$, even when $\epsilon$ is infinitesimally small, since the nearest point to $\tau(\epsilon)$ on $TX$ must be rank three. The tangential variety of the $2 \times 2 \times 2$ Segre variety must thus have significant curvature at its rank two points, which is already suggested by the fact that $\nu$ is tangent to the Segre variety at any tensors in the form $x^1_2 \otimes x^2_1 \otimes y^3_1$ and $x^1_1 \otimes x^2_2 \otimes y^3_1$ for some $y^3_1 \in V^3$. In contrast, the rank three tensors in $TX$ are tangent to distinct point on the Segre variety up to a multiplicative constant. 

\begin{thm}
Let $V^1$, $V^2$, and $V^3$ be two-dimensional real vector spaces. Let $X$ denote the Segre variety of simple tensors in $V^1 \otimes V^2 \otimes V^3$, and let $TX$ denote the tangential variety of $X$. If $\tau \in TX$ is rank three, then it is tangent to a unique point of $X$ up to a multiplicative constant. 
\end{thm}
\begin{proof}
Suppose $\tau$ is tangent to both the tensors ${x^1_1 \otimes x^2_1 \otimes x^3_1}$ and ${x^1_2 \otimes x^2_2 \otimes x^3_2}$. It then follows from Theorem \ref{tangentspace} that 
\begin{align}
    \tau \ & \ = \ y^1_1 \otimes x^2_1 \otimes x^3_1 \ + \ x^1_1 \otimes y^2_1 \otimes x^3_1 \  + \ x^1_1 \otimes x^2_1 \otimes y^3_1, \label{rep_1} \text{ and } \\
    \tau \ &\ = \ y^1_2 \otimes x^2_2 \otimes x^3_2 \ + \ x^1_2 \otimes y^2_2 \otimes x^3_2 \ + \ x^1_2 \otimes x^2_2 \otimes y^3_2, \label{rep_2}
\end{align}
for some vectors $y^i_j$. Since $\tau$ is rank three, the sets $\{ y^i_j, x^i_j\}$ must be independent for all $i,j$. Let $\{y^{i*}_j, x^{i*}_j\}$ be the corresponding dual bases. Considering the contraction maps of $\tau$ with respect to both of these representation (\ref{rep_1}) and (\ref{rep_2}), we conclude that
\begin{align*}
    \Pi_1(\tau)(y^{1*}_1) \ = \ x^2_1 \otimes x^3_1 \ = \ x^2_2 \otimes \left(y^{1*}_1(y^1_2) x^3_2 + y^{1*}_1(x^1_2)y^3_2\right) \ + \ y^{1*}_1(x^1_2)y^2_2 \otimes x^3_2.
\end{align*}
By Lemma \ref{independentunique}, it follows that $y^{1*}_1(x^1_2) = 0$, so $x^2_1 \otimes x^3_1 \ = \ y^{1*}_1(y^1_2) \ x^2_2 \otimes x^3_2$. The fact that $y^{1*}_1(x^1_2) = 0$ implies that $x^1_2 = k_1 x^1_1$ for some constant $k_1$. Furthermore, $\im \Pi_2( x^2_1 \otimes x^3_1 ) $ = $\im \Pi_2(  y^{1*}_1(y^1_2) x^2_2 \otimes x^3_2 )$ implies that $x^2_1 = k_2 x^2_2$ for some constant $k_2$, and $\im \Pi_1( x^2_1 \otimes x^3_1 ) $ = $\im \Pi_1(  y^{1*}_1(y^1_2) x^2_2 \otimes x^3_2 )$ implies that $x^3_2 = k_3 x^3_1$ for some constant $k_3$. Hence, $x^1_2 \otimes x^2_2 \otimes x^3_2$ $= k_1 k_2 k_3 \  x^1_1 \otimes x^2_1 \otimes x^3_1$. 
\end{proof}

\bibliographystyle{unsrt}  

\section*{Acknowledgments}
The author would like to thank Saugata Basu and Ryan Vitale for their valuable discussions about this paper.

\end{document}